\DeclareSymbolFont{AMSb}{U}{msb}{m}{n}
\documentclass[12pt,a4paper,twoside,leqno,noamsfonts]{amsart}
\usepackage[a4paper,top=3.5cm,bottom=3.1cm,left=3cm,right=3cm]{geometry}

\usepackage[utopia]{mathdesign}
\usepackage[utf8]{inputenc}
\usepackage[english]{babel}
\usepackage{amsmath, amsthm, amstext}
\usepackage{microtype}

\usepackage{indentfirst}
\usepackage[colorlinks,bookmarks]{hyperref}
\usepackage{braket}
\usepackage{caption}
\usepackage{stmaryrd}
\usepackage{comment}
\usepackage{multirow}
\usepackage{booktabs}
\usepackage[all]{xy}
\usepackage{graphicx}

\usepackage{listings}

\newtheorem{cor}{Corollary}[section]
\newtheorem*{cor*}{Corollary}
\newtheorem{lem}[cor]{Lemma}
\newtheorem*{lem*}{Lemma}
\newtheorem{thm}[cor]{Theorem}
\newtheorem*{thm*}{Theorem}
\newtheorem{conj}[cor]{Conjecture}
\newtheorem*{conj*}{Conjecture}
\newtheorem{prop}[cor]{Proposition}
\newtheorem*{prop*}{Proposition}

\theoremstyle{definition}
\newtheorem{defn}[cor]{Definition}

\theoremstyle{remark}
\newtheorem{rmk}[cor]{Remark}

\newcommand{\bC}{\mathbb{C}}

\newcommand{\bG}{\mathbb{G}}

\newcommand{\bP}{\mathbb{P}}

\newcommand{\cC}{\mathcal{C}}

\newcommand{\cV}{\mathcal{V}}
\newcommand{\cW}{\mathcal{W}}
\newcommand{\cX}{\mathcal{X}}

\newcommand{\sC}{\mathscr{C}}

\newcommand{\sV}{\mathscr{V}}

\newcommand{\sX}{\mathscr{X}}

\renewcommand{\dim}{\operatorname{dim}}

\newcommand{\ord}{\operatorname{ord}}

\newcommand{\sm}{{\operatorname{sm}}}
\newcommand{\sing}{{\operatorname{sing}}}

\newcommand{\aut}{\operatorname{Aut}}
\renewcommand{\int}{{\operatorname{in}}}
\newcommand{\out}{{\operatorname{out}}}
\newcommand{\red}{{\operatorname{red}}}

\title{A note on non-uniform points for projections of hypersurfaces}

\author{Maria Gioia Cifani}
\address[M.G.C.]{Department of Mathematics and Physics, Roma 3 University, Largo San Murialdo 1, 00146 Roma, Italy}
\email{mariagioia.cifani@uniroma3.it}

\author{Riccardo Moschetti}
\address[R.M]{Department of Mathematics 'G. Peano', University of Turin, via Carlo Alberto 10, 10123 Torino, Italy}
\email{riccardo.moschetti@unito.it}

\begin{document}
\begin{abstract}
Let $X$ be an irreducible, reduced complex projective hypersurface of degree $d$. A \textit{uniform} point for $X$ is a point $P$ such that the projection of $X$ from $P$ has maximal monodromy.
We extend and improve some results concerning the finiteness of the locus of \textit{non-uniform} points for projections of hypersurfaces obtained by the authors and Cuzzucoli in \cite{ccm} only for $P$ not contained in $X$.
\end{abstract}
\maketitle

\section{Introduction}
Let $X$ be an irreducible, reduced complex projective hypersurface of dimension $n\geq 1$ and degree $d$. Fixing a point $P \in \bP^{n+1}$, we can study the linear projection map $\pi_P: \bP^{n+1} \smallsetminus \{P\} \to \bP^n$. If we restrict $\pi_P$ to $X$, we get a finite map of degree $d-1$ or $d$, depending on $P$ being a smooth point of $X$ or being outside $X$. 
It is well known that $M(\pi_P|_X)$ is a subgroup of the symmetric group $S_d$ if $P \notin X$ or $S_{d-1}$ if $P \in X^\sm$. A consequence of the classical uniform position principle is that, when $P$ is general in $\bP^{n+1}$, then $M(\pi_P|_X)$ coincides with $S_d$. 
This motivates the following definition.

\begin{defn} \label{defn: outer and inner uniform points}
A point $Q \in \bP^{n+1} \setminus X$ is called \emph{outer uniform point} if $M(\pi_Q|_X)$ is isomorphic to $S_d$. A point $Q \in X^\sm$ is called \emph{inner uniform point} if $M(\pi_Q|_X)$ is isomorphic to $S_{d-1}$.
We will denote by $\cW(X)^\int$ (resp. $\cW(X)^\out$) the locus of points $Q \in X^\sm$ (resp. $Q \in \bP^{n+1} \setminus X$) which are not--uniform. Furthermore, $\cW(X)$, also called the locus of non--uniform points will be the union of $\cW(X)^\int$ and $\cW(X)^\out$.
\end{defn}

The following conjecture extends \cite[Conjecture 1.7]{ccm} to inner non-uniform points:

\begin{conj} \label{conj:coni}
Let $X$ be a complex irreducible, reduced projective hypersurface of dimension $n\geq 1$. The locus $\cW(X)$ is finite unless $X$ is a cone.
\end{conj}

The case of $n=1$ has been addressed in Proposition 3.4 and Theorem 3.5 of \cite{ps}. Notice that if $X$ is an irreducible curve, $X$ is not a cone and so the conjecture holds. 
Other evidence for the conjecture can be found in the context of Galois points. We say $Q \in \bP^{n+1}$ is a \emph{Galois} point if the field extension associated with the projection $\pi_Q|_X$ is a Galois extension. Galois points are a particular case of non-uniform points, and have been extensively studied in \cite{Yoshihara, FukCom, FukAut}. In particular, Theorem 1.1 and Theorem 1.2 of \cite{FT} show that Conjecture \ref{conj:coni} holds for Galois points when $X$ is a normal hypersurface. 
In a slightly different direction, it has been proven in \cite{cuk} and \cite{cifani2020monodromy} that the conjecture holds true for $X$ general, with $\cW(X)$ empty. The case of $X$ being a smooth hypersurfaces has been proved in \cite[Theorem 1.1]{cms} for $n=2$ and in \cite[Theorem 1.3]{ccm} in higher dimension, both for outer non-uniform points.

The aim of this paper is twofold: we first extend the results of \cite{ccm} to inner non--uniform points, and then we study the possible counterexamples to the conjecture, showing that the only possible class of hypersurface which could provide a counterexample to the conjecture is given by hypersurfaces $X$ in $\bP^{n+1}$, which are not cones, for which there is a $\bP^k$ ($0<k<n$), such that $X\cap \bP^{k+1}$ is reducible for every $\bP^{k+1}$ which contains $\bP^k$. This is done in Corollary \ref{cor:controesempi} and improves \cite[Remark 4.11]{ccm}.

The key result is a generalisation of \cite[Theorem 1.2]{ccm}:
\begin{thm}\label{thm:main}
Let $X$ be a complex irreducible, reduced hypersurface of $\ \bP^{n+1}$, $n \geq 2$. Then $\cW(X)$ is contained in a finite union of linear spaces of codimension $2$ in $\bP^{n+1}$.
\end{thm}

The proof mainly relies on the study of the focal locus of certain families of lines. While the core results were already developed in \cite{CF} and \cite{ccm} the difficulty in generalizing the setting to inner non--uniform point relies upon carefully choosing the family. Moreover, we were able to improve \cite[Theorem 1.2]{ccm} thanks to a local study of the tangent cone to singular points of the hypersurface (Lemma \ref{lem:localcomputationtangentcone}). 

As a consequence of the main theorem, we can show the finiteness of the locus of non--uniform points for smooth hypersurfaces:
\begin{thm} \label{thm:projectionsmooth}
Assume $X$ is a complex smooth projective hypersurface or a general projection of a smooth variety to a hypersurface. Then, the locus $\cW(X)$ is finite.
\end{thm}

We also get a generalisation of the result in \cite{FT}, which also proves the conjecture when $X$ has prime degree.

\begin{thm} \label{thm:transpositions}
Assume $X$ is an irreducible and reduced hypersurface in $\bP^{n+1}$ which is not a cone. 
Then, the number of Galois points is finite. If $X$ has also prime degree (resp. $\deg(X)-1$ is prime), then $\cW(X)^\out$ (resp. $\cW(X)^\int$) is finite. 
\end{thm}

In the following, we will always work on the field $\bC$ of complex numbers. For a variety $X$ we will denote by $X^\sing$ its singular locus, and by $X^\sm$ the locus $X \setminus X^\sing$.

\section{Preliminaries}
\subsection{Families of lines and focal loci}\label{preliminaries}
A family $\cX$ of lines in $\bP^{n+1}$ can be parametrized by an integral subscheme $S$ of the Grassmannian $\bG(1,\bP^{n+1})$. We can describe it by the following diagram

\begin{equation*} 
\xymatrix{
\cX \ar[d]_{p|_{\cX}} \ar[dr]_{f} \ar@{^{(}->}[r]^-i & S \times \bP^{n+1} \ar[d]^{q} & \\
S &  \bP^{n+1},
}
\end{equation*}
where the map $i$ is the inclusion and $p, q$ are the projections on the first and second factors, respectively.

\begin{defn}
The \emph{focal sheaf} is the sheaf over $\cX$ associated with the kernel of the map differential $\mathrm{d}f$. The support of this sheaf is called the \emph{focal locus} of the family $\cX$.
\end{defn}

\begin{defn}
A family $\cX$ of lines in $\bP^{n+1}$ is called \emph{filling family} if the dimension of the base $S$ is $n$ and the map $f=i \circ q$ is dominant. 
\end{defn}

We need some results concerning relationships between the focal locus and the property of being a filling family.

\begin{lem}[\cite{CF}, Proposition 4.1] \label{lem: degree focal scheme}  
Let $\cX$ be a filling family of lines in $\bP^{n+1}$ and let $s$ be a general point of $S$. Then the focal locus in the fibre $\ell_s$ consists of $n$ points counted with the right multiplicity.
\end{lem}

The following result dates back to Segre, in \cite{Segre}:
 
\begin{lem}\label{lem:fundpoint}
Let $\sX$ be a filling family of lines in $\bP^{n+1}$ and assume that $\sX'$ is a subfamily consisting of lines all passing through a point $P$. If the dimension of the base of $\sX'$ is $k$, then $P$ has multiplicity $k$ as a focal point of $\sX$.
\end{lem}

When the integer $k$ used in the previous lemma is greater or equal than $1$, $P$ is also called a fundamental point. We recall the following result on filling families of lines given by the join of two subvarieties.

\begin{lem} \cite[Lemma 3.8]{ccm} \label{lemma:depoi}
Let $F$ be a subvariety of $\bP^{n+1}$ of codimension $2$, and $\sC\nsubseteq F$ be a curve not contained in a $\bP^{n-1}$. Assume that the family $\cX$ of lines joining $\sC$ and $F$ is filling. Then $F$ is linear.
\end{lem}

\subsection{Intersection of lines with a hypersurface} \label{sec:preliminaries multiplicity}
Consider a degree $d$ integral hypersurface $X$ contained in $\bP^{n+1}$, and a line $\ell \nsubseteq X$. Denote by $P_1, \ldots, P_k$ the points in $X \cap \ell$, and by $m_i$ the multiplicity of $P_i$ in $X \cap \ell$. We have that that $\sum m_i=d$.

\begin{defn}
We call the \emph{contact order} of $\ell$ with $X$ at $P_i$ the number $m_i-1$, and we denote it by $\ord_{P_i}(\ell \cap X)$. 
\end{defn}

The line $\ell$ is \emph{transverse} to $X$ at $P_i$ if $\ord_{P_i}(\ell \cap X) = 0$, and \emph{tangent} to $X$ at $P_i$ if $\ord_{P_i}(\ell \cap X) \geq 1$. In the case of higher contact order, i.e. $\ord_{P_i}(\ell \cap X) \geq 2$, we say that the line $\ell$ is an \emph{asymptotic tangent} to $X$ at $P_i$. The line $\ell$ is called \emph{bitangent} to $X$ at two points $P_i \neq P_j$, if $\ell$ is tangent to $X$ at both points $P_i,P_j$. We say that $\ell$ is a \emph{simply tangent} if there is a unique tangent point $P_i\in \ell\cap X$ with $\ord_{P_i}(\ell \cap X) = 1$ and $\ell$ is transverse to $X$ for all the other $P_j\neq P_i$ in $\ell\cap X$. We can classify such behaviours by means of the following datum: 
$$\beta(\ell):=\sum_{P_i \in X \cap \ell}{\ord_{P_i}(\ell \cap X)}.$$ 
We know that $\beta(\ell)=0$ if and only if $\ell$ is transverse to $X$ at all the $P_i$, $\beta(\ell)=1$ if and only if $\ell$ is simply tangent to $X$, and $\beta(\ell)>1$ if and only if $\ell$ is \emph{multi--tangent} to $X$, namely is at least bitangent or asymptotic tangent to $X$ at some points.

\begin{lem}[\cite{ccm}, Proposition 4.1] \label{lem:focalnewtuttoinsieme}
Let $X \subset \bP^{n+1}$ be an integral hypersurface, and $\cX$ be a filling family of lines in $\bP^{n+1}$. Assume the general line $\ell \in \cX$ is tangent to $X$ at a general point $P$. 
Then, the point $P$ is a focus on $\ell$. If the contact order of $\ell$ with $X$ at $P$ is at least $2$, then the multiplicity of $P$ as a focus on $\ell$ is at least $2$.
\end{lem}

We now prove a result using a local computation on the tangent cone that we will use in the proof of the main Theorem. First, we recall the definition of the tangent cone.

\begin{defn}\cite[Lecture 20]{harrisAG} \label{defn:vartanglines}
Consider a hypersurface $X\subset \bP^{n+1}$ and a point $P \in X$. Choose an affine neighborhood of $P$, where $P$ is the origin. Here $X$ is described by the vanishing of a certain polynomial $f:=f_m+f_{m+1}+ \cdots$, where $f_k$ is homogeneous of degree $k$, and $m$ is the smallest integer such that $f_m$ is not identically zero. The \emph{tangent cone} to $X$ at the point $P$, denoted by $C_P(X)$, is the hypersurface of degree $m$ in $\bP^{n+1}$ described by $\{f_m=0\}$.
\end{defn}

\begin{lem} \label{lem:localcomputationtangentcone}
Let $P$ be a singular point of a reduced and irreducible hypersurface $X \subset \bP^{n+1}$. Consider a hyperplane $H$ containing $P$ and a line $\ell \subset H$. If $\ell \in C_P(X)$, then there is a component $Y$ of $X \cap H$ such that $Y^\red$ is tangent to $\ell$ at $P$. 
\end{lem}
\begin{proof}
We restrict ourselves to a general $K \cong \bP^2$, $\ell \subset K \subset H$. We know that $\ell$ is one of the lines of $C_P(X) \cap K$. Since the tangent cone $C_P(X)$ describes $X$ locally around $P$, $C_P(X) \cap K$ describes $X \cap K$ locally around $P$, so there exist a component $Y_K$ of $X \cap K$ such that $Y_K^\red$ is tangent to $\ell$ at $P$ in $K$. As a consequence, we can choose as $Y$ any component  of $X \cap H$ containing $Y_K$.
\end{proof}

\subsection{Monodromy group of projections}
Let $\pi: X \to Y$ a generically finite, dominant morphism of degree $d$, with $X$ and $Y$ irreducible projective varieties. The morphism $\pi$ induces a map
$$\mu: \pi_1(U,y) \to \aut(\pi^{-1}(y)) \simeq S_d$$
where $U$ is a Zariski open subset of $Y$ over which $\pi$ is étale.
\begin{defn}
The image $M(\pi)=\mu \left(\pi_1(U) \right)$ is called \emph{monodromy group} of $\pi$. 
\end{defn}
The monodromy group is isomorphic to the Galois group (\cite[Sec I]{H}), i.e. the Galois group of the field extension $K/k(Y)$, where $K$ is the Galois closure of the field extension $k(X)/k(Y)$. 
In particular, the monodromy group does not depend on the choice of the open set $U$. The group $M(\pi_P)$ is a transitive subgroup of the symmetric group since $X$ is irreducible. We recall some standard facts on transitive permutation groups (see for instance \cite[Ch.8]{Isaacs}). 
\begin{prop} \label{prop:generatedtranspositions}
Let $G$ be a transitive subgroup of $S_d$. Then the following hold
\begin{itemize}
    \item if $G$ is finitely generated and its generators are all transpositions but one, then $G=S_d$;
    \item if $G$ is primitive and contains a transposition then $G=S_d$. 
\end{itemize}
\end{prop}
\begin{rmk}\label{remark:decomposable}
A map $\pi_P$ is \emph{decomposable} if it factors non birationally over a open subset, see \cite[Definition 2.1]{ps}. If $\pi_P$ is indecomposable, then the monodromy group $M(\pi_P)$ is primitive, see \cite[Remark 2.2]{ps}. 
\end{rmk}
The following Bertini-type result will be used various times in the following. It can be useful for instance to find a set of generators for the monodromy group.

\begin{prop} \label{prop:monodromygeneralsection}
Let $X$ be an irreducible variety in $\bP^{n+1}$ of dimension $n\geq 2$. Let $H \simeq \bP^2$ be a general plane and $X_H:= X \cap H$. Then, for a point $P\in H$ we have
\begin{equation*} 
\xymatrix{
X_H \ar[d]_{{\pi_P}_{|X_H}}  \ar@{^{(}->}[r]^i & X \ar[d]^{\pi_P} & \\
\bP^1  \ar@{^{(}->}[r]^i &  \bP^n.
}
\end{equation*}
As a consequence we have $M({\pi_P}_{|X_H}) \leq M(\pi_P)$, which in term of non-uniform points gives that
$$\cW(X) \cap H \subseteq \cW(X_H).$$
\end{prop}

The construction of the monodromy group associates a permutation to a generator of the fundamental group, which, thanks to the previous proposition can be associated with a branch point. The following proposition describes the cycle structure of the permutation depending on the structure of the fibre of the projection.

\begin{prop}\cite[Prop. 2.5]{ps}\label{prop:cyclestructure}
If $\sum_{j=1}^k m_jP_j$ denotes the scheme theoretic fibre of $\pi$ over a branch point, then the associated permutation have a cycle structure $(m_1,\ldots, m_k)$.
\end{prop}

In particular, simple branch points correspond to transpositions in the monodromy group \cite[Lemma, pag. 698]{H}.

By combining Proposition \ref{prop:generatedtranspositions} and Proposition \ref{prop:cyclestructure} we get the following property of non--uniform points in the case of planar curves, that we will use in the proof of the main results.   

\begin{lem}\label{prop:semplice 2 multitangent}
Let $C \subset \bP^2$ be an irreducible and reduced curve and. If $P \in \cW(C)$ is a non--uniform point, then there are at least two lines multi--tangent to $C$ passing through $P$. 
\end{lem}

\begin{rmk}\label{rmk:tangentisemplici}
When $P \in X^\sm$ the map $\pi_P|_X$ is rational and not defined on $P$. We compose it with the blow up $\nu: Bl_PX \to X$ to get a morphism $Bl_PX \stackrel{\pi}{\to} \bP^n$ and the monodromy of $\pi_P|_X$ coincides with the monodromy of $\pi$. 
In particular, %if the scheme theoretic fibre of $\pi_P|_X$ over a branch point is $2P+m_2P_2+\ldots + m_kP_k$, with $P_j\neq P,\ \forall j=2,\ldots,k$, then the associated permutation has a cycle structure $(1,m_2,\ldots,m_k)$. In other words, 
a line bitangent to $X$ being simply tangent at $P$ and in another point is associated to a transposition in the monodromy group of the projection $\pi_P$. 
\end{rmk}

\section{The locus of non--uniform points}
This section is devoted to the proof of the main results of this paper. As a quick preamble, we recall that there are two differences with respect to the arguments of \cite{ccm}: first, \cite[Theorem 1.2]{ccm} is improved, removing a rather technical case. This is a step forward in the direction of proving Conjecture \ref{conj:coni}, and also simplify the proofs of Theorem \ref{thm:projectionsmooth}, \ref{thm:transpositions}. Second, all the results are generalized also to $\cW(X)^\int$: we now state them with $\cW(X)$, which is the union of $\cW(X)^\out$ and $\cW(X)^\int$, and when it is required, we handle the two cases separately in the proofs.

We will need to work with a family of lines that are multi--tangent to $X$. For the technical part of the proof, we will need to select an appropriate subfamily, which we will call $\cV(X)$. It will be related to the lines corresponding to generators of the non-uniform monodromy groups as in Proposition \ref{prop:cyclestructure}. We will denote by $\cV(X)_P$ the family of multi--tangent lines to $X$ passing through $P$ which are still multi tangent at $X$ after removing the potential contribution from the point $P$. In the notation of Section \ref{sec:preliminaries multiplicity}, consider a line $\ell$ multi--tangent to $X$, so $\beta(\ell)>1$. If $P \notin X \cap \ell$, then trivially $P$ has no contribution in $\beta(\ell)$, so $\ell \in \cV(X)_P$. If $P \in X \cap \ell$, then we have $\ell \in \cV(X)_P$ if $$\sum_{P_i \in X \cap \ell, P_i \neq P}{\ord_{P_i}(\ell \cap X)}>1.$$

\begin{defn}
Let $\cV(X) \subset \bG(1,\bP^{n+1})$ be the family of lines in $\bP^{n+1}$ obtained by the union of $\cV(X)_P$ for $P \in \cW(X)$. If $\sC \subset \bP^{n+1}$ is an algebraic set, we will denote by $\cV_\sC$ the union of $\cV_P(X)$ for $P \in \sC$.
\end{defn}

We want to apply the results about the focal loci of the family $\cV(X)_\sC$, when $\sC$ is a special curve contained in $\cW(X)$. We remark that by our definition of tangency, the lines of $\cV(X)_P$ are not contained in $X$. The fact that the family is non--empty comes from the next two lemmas.

\begin{lem}
Let $X$ be an irreducible, reduced hypersurface in $\bP^{n+1}$ of degree greater than $1$. Let $P \in X^\sm$. Then the family of lines passing by $P$ and contained in $X$ has dimension smaller than $n-1$.
\end{lem}
\begin{proof}
Assume by contradiction that the dimension is $n-1$. Then $X$ contains a $\bP^n$ or a cone with vertex $P$. This is not possible because $X$ is irreducible of degree greater than $1$.
\end{proof}

\begin{lem} \label{lem:dimension}
Let $X$ be an irreducible, reduced hypersurface of $\ \bP^{n+1}$, $n \geq 1$. Assume $Q \in \cW(X)$ general, Then the family $\cV_Q$ has dimension $n-1$.
\end{lem}
\begin{proof}
We will proceed by induction on $n$, starting from the case $n=1$. By applying Lemma \ref{prop:semplice 2 multitangent} we know that there must be at least two multi--tangent lines to $X$ and passing by $Q$. If $Q \in \cW(X)^\int$, notice that one of the two lines must be different from $T_Q(X)$, hence the dimension is still $0$.
For the general case $n>2$, assume the result is true for $X$ of dimension $n-1$, and prove it for $X$ of dimension $n$. Let $Q \in \cW(X)$ general, assume by contradiction $\dim \cV_Q(X) < n-1$ and take a general hyperplane $H$ such that $Q \in H$. We will get $\dim \cV_Q(X \cap H) < n-2$, and this is in contradiction with the induction hypothesis.
\end{proof}

Now we need to describe when the family $\cV(X)_\sC$ is filling.

\begin{prop} \label{prop:filling}
Let $X$ be an irreducible, reduced hypersurface of $\ \bP^{n+1}$, $n \geq 2$. Assume $\sC \subset \cW(X)$ is an irreducible curve which is not contained in a linear space of codimension $2$. Then $\cV(X)_\sC$ is filling.
\end{prop}
\begin{proof}
Thanks to Lemma \ref{lem:dimension}, we know that $\cV_\sC$ has dimension $n$. It remains to prove that $\cV_\sC$ is dominant. Assume by contradiction that this is not the case, so that $\cV_\sC$ is contained in a finite union of hypersurfaces $V_j \subset \bP^{n+1}$, $j=1, \dots, r$. If we consider $Q \in \sC$ general, $\cV_Q(X)$ is the union of the cones over $X \cap V_j$ with vertex $Q$, each of which correspond to $V_j$. 
As a consequence we vary the vertex $Q \in \sC$, $\cV_Q(X)$ stays the same, so each component $V_j$ is constant. Hence, the whole $\sC$ is contained in the vertex of each $V_j$, and by \cite[Proposition 1.3 (i)]{Adlandsvik}, also the linear span $\langle \sC \rangle$ must be contained in the vertex of each $V_j$. Since we are assuming that the codimension of $\langle \sC \rangle$ is strictly greater than $2$, the only possibility is that each $V_j$ is a linear space. Moreover, since $\sC$ is contained in each $V_j$, there must be only one of them, so $r=1$, and $\sC$ is a curve in a $\bP^n$.

Assume first that $n>2$, take a general point $Q \in \sC$ and a general plane $K \cong \bP^2$ such that $Q \in K$. By the genericity of the choice of $K$, we know that the family $\cV_Q(X \cap K)$ is composed by just the line $V_1 \cap K$. If $Q \in \cW(X)^\out$, then we just notice that by Lemma \ref{prop:semplice 2 multitangent} there must be at least two lines on $\cV_Q(X \cap K)$, and this gives a contradiction. If $Q \in \cW(X)^\int$, we have to be more careful due to the way we defined $\cV_Q(X)$. Since $K$ is general, we know that the tangent line $C_Q(X \cap K)$ is simply tangent at $Q$, hence it cannot be one of the two lines described in \ref{prop:semplice 2 multitangent}, as discussed in Remark \ref{rmk:tangentisemplici}. As a consequence, we get again that there are at least two lines on $\cV_Q(X \cap K)$, and this gives a contradiction.

We are left with the case $n=2$. In this case, we know that the tangent line $C_Q(X \cap K)$ is simply tangent at $Q$ because we chose $Q$ general in $\cC$.
\end{proof}

\begin{rmk} \label{rem:cases}
Consider a curve $\sC \subset \cW(X)$ and a general line $\ell \in \cV(X)_\sC$. Since the lines of this family are multi tangent at $X$, we have one of the following behaviours
\begin{enumerate}
    \item[(C1)] The line $\ell$ if is bitangent or asymptotic tangent to $X^\sm$;
    \item[(C2)] The line $\ell$ passes through a point of $X^\sing$ and is also tangent at a (necessarily different) point of $X^\sm$;
    \item[(C3)] The line $\ell$ intersects $X^\sing$ in more than one point;
    \item[(C4)] The line $\ell$ is in the tangent cone to X at a point in $X^\sing$.
\end{enumerate}
Moreover, from the construction of $\cV(X)$ it is clear that all the tangency points are not in $\sC$.
\end{rmk}

We now prove the generalisation of \cite[Lemma 4.5]{ccm}. While the proof follows the same lines, we report it here for the reader convenience to explicitly show that it works for the case of inner non-uniform points. 

\begin{lem} \label{lem:case123}
Let $X$ be an irreducible, reduced hypersurface of $\bP^{n+1}$, $n \geq 2$. Consider the family $\cV_\sC$ associated with curve $\sC \subset \cW(X)$. If $\sV_\sC$ has a maximal component of lines of type $(C1)$, $(C2)$ or $(C3)$, then $\sC$ is contained in a linear space of codimension $2$ in $\bP^{n+1}$. 
\end{lem}
\begin{proof}
Assume by contradiction $\sC$ is not contained in a linear space of codimension $2$ in $\bP^{n+1}$. Then $\cV_\sC$ is filling by Proposition \ref{prop:filling}. Consider a general $\ell$ in $\cV_\sC$. By construction, there is a point $Q \in \ell \cap \sC$. By Lemma \ref{lem:dimension} the family $\cV_Q$ has dimension $n-1$, so by Lemma \ref {lem:fundpoint} $Q$ is a focal point on $\ell$ of multiplicity at least $n-1$.
We know that $\ell$ is multi--tangent at $X$ in points outside $Q$. This comes from the definition of $\cV(X)$, so depending on the case in analysis, we can have the following:
\begin{enumerate}
    \item[(C1)] Apply Lemma \ref{lem:focalnewtuttoinsieme}. If the line $\ell$ is bitangent to $X^\sm$ at two different points, they are foci of multiplicity at least $1$. If the line $\ell$ is asymptotic tangent to $X^\sm$, it is a focus of multiplicity at least $2$;
    \item[(C2)] By Lemma \ref{lem:focalnewtuttoinsieme}, the tangent point of $X^\sm$ is a focus of multiplicity at least $1$.  By a dimension count, there is a subfamily of $\cV_\sC$ of dimension at least $1$ of lines passing through the point of $X^\sing$. By Lemma \ref{lem:fundpoint}, it is a focus of multiplicity at least $1$.
    \item[(C3)] By a dimension count and by Lemma \ref{lem:fundpoint}, both singular points are foci of multiplicity at least $1$.
\end{enumerate}

In each case, the focal locus on $\ell$ contains the point $Q$ with multiplicity at least $n-1$, and other points which contribute with multiplicity at least $2$. So the degree of the focal scheme on $\ell$ is at least $n+1$, but this contradicts Lemma \ref{lem: degree focal scheme}.
\end{proof}

We are now ready to prove the main result of this paper.

\begin{proof}[Proof of Theorem \ref{thm:main}]
The case $n=2$ is an easy consequence of the results of \cite{ps}. Let now $n>2$ and assume that a component of $\cW(X)$ is not contained in a linear spaces of codimension $2$ in $\bP^{n+1}$. Consider a curve $\sC \subset \cW(X)$ with the same property. Thanks to Proposition \ref{prop:filling} the family ${\cV_\sC}$ is filling. Let $Q$ be a general point in $\sC$ and $\ell$ be the general line in ${\cV_Q}$. Since the dimension of ${\cV_Q}$ is $n-1$, the point $Q$ is a focus of multiplicity $n-1$ on $\ell$ by Lemma \ref{lem:fundpoint}. As stated in the Remark \ref{rem:cases}, we know that $\ell$ is multi tangent to $X$ in points different from $Q$, which also are part of the focal scheme induced by ${\cV_\sC}$ on $\ell$. If we are in the cases (C1), (C2), (C3) of Remark \ref{rem:cases}, we can apply Lemma \ref{lem:case123} and get a contradiction.

It remains to handle the case (C4), namely $\ell$ belongs to the tangent cone of a point in $X^\sing$. The family $\widetilde{\cV_\sC}$ is composed by lines joining $\sC$ and $X^\sing$. By Lemma \ref{lemma:depoi} we get that $X^\sing$ is the union of a finite number of linear components each isomorphic to $\bP^{n-1}$. Moreover, we are assuming the general line belongs to the tangent cone of a point in $X^\sing$, hence, in order to have a non-uniform monodromy, the singular locus must split in at least two linear spaces (Proposition \ref{prop:generatedtranspositions}). Let $X_1$ be one of the components of $X^\sing$, and take a general hyperplane $H \cong \bP^n$ passing through $X_1$. Due to the fact that $\widetilde{\cV_\sC}$ is filling, we can choose $X_1$ such that $\sC$ is not contained in a hyperplane passing through $X_1$, hence, on $H$ there is at least one point $Q \in \sC$ not contained in $X_1$. If we take $x \in X_1$ general, the line $\langle x, Q \rangle$ belongs to the tangent cone $C_x(X)$. By Lemma \ref{lem:localcomputationtangentcone}, there must be a component of the intersection $X \cap H$ passing through $x$ which is different from $X_1$, and this is a contradiction.
\end{proof}

We can now prove some results that are consequences of the main Theorem. Theorem \ref{thm:projectionsmooth} shows that Conjecture \ref{conj:coni} holds with the additional hypothesis of $X$ being smooth or a general projection of a smooth variety. We report the Bertini theorem as in \cite[Theorem 3.3.1]{Laz1} formulated in the context we will need for the proof of Theorem \ref{thm:projectionsmooth}.
\begin{thm} \label{thm:bertinilaz}
Let $\tilde{X} \subset \bP^{2n+1}$ a variety of dimension $n$. Let $\Lambda$ be a linear space of dimension $2n-d$ such that $\tilde{X} \cap \Lambda$ is smooth at every point. Then the section $\tilde{X} \cap H$, for a general $H$ containing $\Lambda$ is irreducible.
\end{thm}

The proof of Theorem \ref{thm:projectionsmooth} relies on a famous result of \cite{Mather}, which is very useful to study the singularities arising from projection maps. The following is an algebraic version which is written here restricting it to the case we will need. The complete version is \cite[Theorem 1, Theorem 2]{abo}.

\begin{thm}\cite[Theorem 1]{abo} \label{thm:Mather}
Let $\tilde{X} \subset \bP^{2n+1}$ be a smooth variety of dimension $n$. 
Let $L$ be a linear subspace of dimension $t = n-1$ such that $L \cap X= \emptyset$. 
Let $X$ be the image of $\tilde{X}$ under the linear projection $\pi_L\colon \bP^{2n+1} \to \bP^{n+1}$. 
For any $i \leq n$, define $\tilde X_{i}:=\{x \in \tilde X\ |\ \dim(T_x \tilde X \cap L)=i-1\}$. 
For $L$ general, every $\tilde X_{i}$ is smooth and, when not empty, its codimension in $\tilde X$ is $i \cdot (i + 1)$.
\end{thm}

We are now ready to prove Theorem \ref{thm:projectionsmooth}.

\begin{proof}[Proof of Theorem \ref{thm:projectionsmooth}]
Fix $n = \dim X$, and assume that $X \subset \bP^{n+1}$ is a general projection of a smooth variety $\tilde X$. By \cite[Section 5]{FulLaz}, we can assume $\tilde X \subset \bP^{2n+1}$.
Assume by contradiction that $\cW(X)$ is not finite, and denote by $K\cong \bP^k$ the smallest linear subspace of $\bP^{n+1}$ containing an irreducible component of $\cW(X)$. We remark that by Theorem \ref{thm:main}, we have $1 \leq k \leq n-1$. Consider a general $H\cong\bP^{k+1}$ containing $K$. 

Assume first $X \cap H$ is irreducible. By hypothesis $X$ is reduced, and the base locus of the linear system $X \cap H$ is $X \cap K$. Notice that $\dim(X \cap K) = \dim(X \cap H) -1$, unless $K \subset X$, and so $X \cap H$ is reducible.
As a consequence, we get that $X \cap H$ is reduced, see for instance \cite[Proposition 4.6.1]{MR199181}. Now apply Proposition \ref{prop:filling} to a curve $\sC \subset \cW(X)^\out$ which spans $K$. Since $\sC$ spans a subspace of codimension $1$ in $H$, we get that the family of lines in $\cV_{\sC}$ which belongs to $H$ is filling. Since $X \cap H$ is reduced and irreducible, hypothesis of Theorem \ref{thm:main} are satisfied and we immediately get a contradiction. 

Now assume $X \cap H$ is reducible. Denote by $L$ the linear space in $\bP^{2n+1}$ which gives the projection $\pi_L:\tilde{X} \to X$. Let also $\tilde{H}:=\langle L, H \rangle$, and $\tilde{K}:=\langle L,K\rangle$. Notice that $\tilde{K} \cong \bP^{k+n}$ is a suitable $\Lambda$ for applying \ref{thm:bertinilaz}. Since we are assuming $X \cap H$ is reducible, we get that $\tilde{X} \cap \tilde{K}$ must be not reduced. This means that $T_x{\tilde X} \subset \tilde K$ for $x \in \tilde{X} \cap \tilde{K}$. 
As a consequence, the dimension of $T_x \tilde X \cap L$ in $\tilde K$ is $n-k-1$. From the fact that $1 \leq k \leq n-1$, we have $n-k-1\geq 0$. 
This shows that the points of $\tilde X \cap \tilde K$ are contained in $\tilde X_{n-k}$, defined in Theorem \ref{thm:Mather}.
Now, $\tilde X \cap \tilde K$ has dimension at least $k-1$ and $\tilde X_{n-k}$ has dimension $n-(n-k)(n-k+1)$. We get
\begin{align*}
    k-1 & \leq n-(n-k)(n-k+1)\\
    0 &\leq 1-(n-k)^2
\end{align*}
It follows that $k=n-1$, and so $\dim (\tilde X \cap \tilde K) = \dim \tilde X_1$. Since $\tilde X_1$ is smooth by Theorem \ref{thm:Mather}, this gives a contradiction.
\end{proof}

Theorem \ref{thm:transpositions} is a consequence of the following result about transpositions in the monodromy group associated with non--uniform points. The idea is to try to study the situation arising from possible counterexamples of Conjecture \ref{conj:coni}.
 
\begin{prop}\label{prop:trasposizioni}
Let $X$ be an irreducible, reduced hypersurface of $\bP^{n+1}$. Assume $\cW(X)$ is not finite, and $X$ is not a cone, then the monodromy group associated with all but finitely many points of $\cW(X)$ contains transpositions.
\end{prop}

\begin{proof}
This result has been proved for outer non--uniform points in \cite[Theorem 1.4]{ccm}. For simplicity here we will restrict to the case of inner non--uniform points.
Let $\sC$ be a curve inside $\cW(X)^\int$ and let $\cX_Q$ be the family of lines through a point $Q \in \sC$ tangent to $X$ at smooth points and not lying in the tangent hyperplane to $X$ at $Q$. Let $\cX$ be the union of $\cX_Q$ for every $Q \in \sC$. 

Let $X^* \subset (\bP^{n+1})^*$ be the dual variety of $X$ and let $r$ be the dimension of $X^*$. If $X$ is not a cone, by \cite[Theorem 1.25]{ProjDual} we have that $X^*$ is not contained in a hyperplane. Consider the family of hyperplanes $\{H_Q^{\vee}\}$ in $(\bP^{n+1})^*$ dual to the points of $Q \in \sC$. 

If $H_Q^{\vee}$ is general in this family, $\dim(X^* \cap H_Q^{\vee})=r-1$. Since every point of $X^*$ is contained in one of such hyperplanes, by dualization we get that the tangent hyperplane to $X$ at a general point pass through the general point of $\sC$.
As a consequence, the general $Q \in \sC$ is contained in a $r-1$ dimensional family of tangent hyperplanes to $X$. 
The general element of this family is tangent to $X$ along a subvariety of dimension $n-r$ and every line joining $Q$ and this subvariety belongs to $\cX$. In particular, the dimension of the family $\cX$ is $n$, and the general line in $\cX$ is tangent at the general point of $X$.

We still need to show that $\cX \to \bP^{n+1}$ is dominant. If this were not the case, the general line of $\cX$ would be contained in a finite union of hypersurfaces $V_i$, hence it would not be tangent at the general point of $X$. Notice that $V_i$ cannot coincide with $X$ itself since $X$ is not a cone. 

Consider the focal scheme of $\cX$ on $\ell$. Since $\ell \cap Q$ is a point of multiplicity $n-1$, it gives by Lemma \ref{lem:fundpoint} a focus of multiplicity $n-1$. Assume that the general line $\ell$ of $\cX$ is not simply tangent to $X$. Then, we would have on $\ell$ at least two focal points of degree one, or at least one focal point of degree $2$. This contradicts Lemma \ref{lem: degree focal scheme}.

As a consequence, if $X$ is not a cone, we can find simple tangent lines to $X$ passing through the general point of $\cW(X)^\int$, that correspond to transpositions in the monodromy group. 
\end{proof}
As a consequence, we can prove Theorem \ref{thm:transpositions}: 

\begin{proof}[Proof of Theorem \ref{thm:transpositions}]
Proposition \ref{prop:trasposizioni} implies that, if $\cW(X)$ is not finite and $X$ is not a cone, then for all but finitely many points there exist a transposition in the monodromy group. As a consequence, for all but finitely many points $Q \in \cW(X)$, the field extension given by the projection from $Q$ is not Galois, since the action of an element in the Galois group on a general fibre has no fixed components. Thus, the locus of Galois points is finite. 

Morevoer, if the map $\pi_P$ is indecomposable and the monodromy group of $\pi_P$ contains a transposition, then the point $P$ is uniform, see Remark  \ref{remark:decomposable}. As a consequence, if $P \notin X$ and $\deg(X)$ is prime, then the map $\pi_P$ is indecomposable and so $\cW(X)^\out $ is finite. If $P \in X^\sm$ and $\deg(X)-1$ is prime, then the map $\pi_P$ is indecomposable and so $\cW(X)^\int$ is finite. 
\end{proof}
 
We conclude this section by some considerations on Conjecture \ref{conj:coni}, which at the moment is the state of the art generalisation of \cite[Remark 4.11]{ccm}.
 
\begin{cor}\label{cor:controesempi}
The only class of hypersurface which could provide a counterexample to conjecture \ref{conj:coni} is given by irreducible and reduced hypersurfaces $X$ in $\bP^{n+1}$, which are not cones, and such that there exists a $\bP^k$, $0<k<n$, where $X\cap \bP^{k+1}$ is reducible for every $\bP^{k+1} \supset \bP^k$. 
\end{cor}
\begin{proof}
By Theorem \ref{thm:main} the non--uniform locus $\cW(X)$ is not finite and contained in a linear subspace $\langle \cW(X) \rangle \cong \bP^k$ with $k<n$. Consider a general $H\cong \bP^{k+1}$ such that $H \supset \bP^k$. The section $X \cap H$ is reduced since $X$ is. If moreover $X \cap H$ is irreducible, then we get a contradiction to Theorem \ref{thm:main}, since $\cW(X)=\cW(X \cap H)$ is in codimension one in $H$. 
\end{proof}
 
\section*{Acknowledgements}
The authors are members of GNSAGA of INdAM. During the preparation of the paper, the authors were partially supported by PRIN 2017 Moduli and Lie theory, and by MIUR: Dipartimenti di Eccellenza Program (2018-2022) -
Dept. of Math. Univ. of Pavia.

We would like to thank Gian Pietro Pirola for the helpful discussions and suggestions during the preparation of the paper. 
We are also grateful to Alberto Albano and Francesco Russo for the help in understanding how to remove the condition (C4) from the proof of the main result.

\bibliographystyle{siam}
\bibliography{biblio}
 
\end{document}